\documentclass[11pt]{amsart}

\usepackage{amsmath,amssymb,amscd,verbatim}

\newcommand{\F}{\mc F}

\newcommand{\cS}{\mc S}

\newcommand{\mc}{\mathcal}

\newcommand{\sub}{\subseteq}

\newcommand{\nsub}{\nsubseteq}

\newcommand{\lra}{\Leftrightarrow}

\newcommand{\ra}{\Rightarrow}

\newcommand{\sm}{\setminus}

\newcommand{\tmax}{t\op{-Max}}

\newcommand{\Na}{\op{Na}}
\newcommand{\spec}{\op{Spec}}

\newcommand{\astmax}{\ast\op{-Max}}

\newcommand{\op}{\operatorname}

\newtheorem{theorem}{Theorem}[section]

\newtheorem{lemma}[theorem]{Lemma}

\newtheorem{proposition}[theorem]{Proposition}

\newtheorem{corollary}[theorem]{Corollary}

\newtheorem{remark}[theorem]{Remark}

\theoremstyle{definition}

\newtheorem{Qu}{Problem}

\begin{document}

	\title{$w$-stability and Clifford $w$-regularity of polynomial rings}



\author{Stefania Gabelli and Giampaolo Picozza}

\address{Dipartimento di Matematica e Fisica, Universit\`{a} degli Studi Roma
Tre,
Largo S.  L.  Murialdo,
1, 00146 Roma, Italy}

\email{gabelli@mat.uniroma3.it}

\email{picozza@mat.uniroma3.it}

\date{\today}


\begin{abstract}
We investigate the transfer of $w$-stability  and Clifford $w$-regularity from a domain $D$ to the polynomial ring $D[X]$. We show that these two properties pass from $D$ to $D[X]$ when $D$ is either integrally closed or it is Mori and $w$-divisorial.

\end{abstract}

\maketitle


\section{Introduction}

The transfer of properties from a ring $D$ to the polynomial ring $D[X]$ is an important subject of study in commutative algebra.
 A basic result in this direction is  Hilbert's Basis Theorem, which states that a polynomial ring over a Noetherian ring is  still Noetherian.
However,  several good properties of  classical domains of the ideal theory do not pass to the polynomial ring. For example,  the ring $\mathbb{Z}$ of the integers  is a principal ideal domain (for short, a PID), so it is a Dedekind domain, a Bezout domain and a Pr\"ufer domain. But  it is easily seen that the ring $\mathbb{Z}[X]$ has none of these properties.

More recently, several other classes of domains have been studied in mutiplicative ideal theory; for example, divisorial domains, stable domains and Clifford regular domains. Again, $\mathbb{Z}$ has all of these properties but none of them  passes to $\mathbb{Z}[X]$.

In fact, something much stronger is true: a polynomial ring over a domain that is not a field is never a PID, Dedekind, Bezout, Pr\"ufer, divisorial, stable or Clifford regular domain. The main obstruction is that all these classes of domains are in the class of DW-domains.

A DW-domain is a domain in which each nonzero ideal is a $w$-ideal (i.e., a semidivisorial ideal, following Glaz and Vasconcelos \cite{GV}) or, equivalently, a domain in which every maximal ideal is a $t$-ideal (see for example \cite[Proposition 2.2]{M1}).
Houston and Zafrullah proved that the $t$-maximal ideals of $D[X]$ are either uppers to zero or are extended from ideals of $D$ \cite[Proposition 1.1]{HZ}. Since in polynomial rings over domains that are not fields, there are always maximal ideals that are neither uppers to zero nor extended, then the polynomial ring $D[X]$ is never a DW-domain. That is, as shown with a direct proof by Mimouni, the following theorem holds:

\begin{theorem} \label{th1} \cite[Proposition 2.12]{M1}
Let $D$ be an integral domain. Then $D[X]$ is a DW-domain if and only if $D$ is a field.
\end{theorem}

All classes of domains mentioned above have been generalized, by requiring that the ideal theoretic properties that define these domains hold on $w$-ideals and not necessarily on the set of all nonzero ideals.
For example,  a  Pr\"ufer $v$-multiplication domain (for short, a P$v$MD) is a domain in which every localization at a $w$-maximal ideal is a valuation domain; thus we can say that  a P$v$MD is the $w$-version of a Pr\"ufer domain. A strong Mori domain is a domain which satisfies the ascending chain condition on $w$-ideals; hence, it is the $w$-version of a Noetherian domain. Similarly, a Krull domain is a strong Mori P$v$MD and so it is the $w$-version of a Dedekind domain.

Moreover, a $w$-principal domain is a domain in which every $w$-ideal is principal, and a $w$-Bezout domain is a domain in which every $w$-finite ideal (i.e., an ideal that is the $w$-closure of a finitely generated ideal) is principal   \cite{E2}: these two notions respectively generalize the notions of principal and Bezout domain. It is easy to see that the class of $w$-Bezout domains coincides with the class of domains with the greatest common divisor (for short, GCD-domains) \cite[Theorem 3.3]{E2}, while the class of $w$-principal domains coincides with the class of unique factorization domains (for short, UFDs) \cite[Theorem 2.5]{E2}. Summarizing, we have the following:

\begin{theorem} Let $D$ be an integral domain. Then:
\begin{itemize}
\item[(1)] $D$ is a Pr\"ufer domain if and only if it is a P$v$MD and a DW-domain.
\item[(2)] $D$ is a Dedekind domain if and only if it is a  Krull DW-domain.
\item[(3)] $D$ is a Bezout domain if and only if it is a $w$-Bezout DW-domain (i.e., a GCD-domain and a DW-domain).
\item[(4)] $D$ is a PID if and only if it is a $w$-PID and a DW-domain (i.e., a UFD and a DW-domain).
\end{itemize}
\end{theorem}

What is interesting is that, at least for these classical domains (PIDs, Dedekind, Bezout and Pr\"ufer domains), the fact of being a DW-domain is the ``only'' obstruction to the transfer of the property to the polynomial ring. Indeed, it is well known that the properties of being  a UFD, a Krull domain, a GCD-domain and  a P$v$MD extend to  polynomial rings.

So, roughly speaking, we can say that things go well, for ``classical rings'', when one removes the DW-property from the definition.

It is natural to ask whether the same happens for the classes of domains introduced more recently, such as divisorial, stable and Clifford regular domains.

All these three notions have been generalized by using the $w$-operation. For example, recall that a divisorial domain is a domain in which all  nonzero ideals are divisorial \cite{BS2}. A $w$-divisorial domain was defined by El Baghdadi and  Gabelli as a domain in which the $w$-operation coincides with the $v$-operation \cite{GE}.  Again,  $w$-divisorial domains can be considered as the $w$-version of divisorial domains.
 The transfer of $w$-divisoriality to polynomial rings was studied in \cite{GHP} by the  authors of this paper and E. Houston. We showed that with good hypotheses on $D$ (for example, if $D$ is integrally closed or Mori),  $w$-divisoriality passes to polynomial rings. The general case is still open and it is somehow linked with Heinzer's old conjecture about the integral closure of divisorial domains.

Recently, we have defined $w$-stable domains and Clifford $w$-regular domains  as natural generalizations of stable and Clifford regular domains  \cite{GP, GP1}. In the present paper, we start a study of polynomial rings over $w$-stable  and Clifford $w$-regular domains.

First, we observe that the $w$-stability (resp., Clifford $w$-regularity) of $D$ is a necessary condition for the $w$-stability (resp., Clifford $w$-regularity) of $D[X]$
and we show that determining whether this condition is also sufficient depends only on the ideals of $D[X]$ that are not extended from $D$. Moreover, by using the local characterization of $w$-stable domains, we prove that the $w$-stability of $D[X]$ is equivalent to the stability of the $v$-Nagata ring of $D$. An analogous result for Clifford $w$-regularity needs some additional hypotheses. Finally, we give two positive results. Namely, we show that if $D$ is integrally closed, then $D$ is $w$-stable (resp., Clifford $w$-regular) if and only if $D[X]$ is $w$-stable (resp., Clifford $w$-regular), if and only if the $v$-Nagata ring over $D$ is stable (resp., Clifford regular). We also prove  that $w$-stability and Clifford $w$-regularity  are equivalent for Mori domains and that if $D$ is Mori,  $D$ is $w$-stable and $w$-divisorial if and only if $D[X]$ is $w$-stable and $w$-divisorial if and only if the $v$-Nagata ring over $D$ is totally divisorial.

\section{Preliminaries}
Throughout this paper, $D$ will be an integral domain and $K$ its field of fractions. To avoid trivialities, we will assume that $D\neq K$. A \emph{local} domain is a domain with a unique maximal ideal, not necessarily Noetherian.
An \emph{overring} of $D$ is a domain $T$ such that $D\sub T\sub K$.
If $I$ is a fractional ideal of $D$, we call $I$ simply an \emph{ideal} and if $I\sub D$ we say that $I$ is an \emph{integral ideal}.

\subsection{Star operations}
Divisorial ideals, $t$-ideals and $w$-ideals are examples of \emph{star ideals}, that is ideals closed under a \emph{star operation} \cite[Section 32]{g1}.

A {\it star operation} is
a map $I\to I^\ast$  from the set ${\mc F} (D)$ of nonzero ideals of $D$ to
itself such that:

(1) $D^\ast = D$ and $(aI)^\ast = aI^\ast$, for all $a \in K \smallsetminus
\{0\}$;

(2) $I \sub I^\ast$ and $I \sub J \ra I^\ast \sub J^\ast$;

(3) $I^{\ast\ast} = I^\ast$.

 A nonzero ideal $I$ such that  $I=I^\ast$ is called a \emph{$\ast$-ideal}. Nonzero principal ideals are $\ast$-ideals.

 A star operation $\ast$ is of {\it finite type} if
 $I^\ast = \bigcup \{J^\ast\, ;\, J \sub I$  and $J$ is fini\-te\-ly
ge\-ne\-ra\-ted\},
 for each $I \in {\mc F} (D)$.
To any  star operation $\ast$, we can associate a star operation
$\ast_{f}$ of finite type
by defining $I^{\ast_{f}}= \bigcup J^\ast$, with the
union taken over all finitely generated ideals $J$ contained
in $I$. Clearly $I^{\ast_{f}} \sub I^\ast$ and $J^{\ast_{f}} = J^\ast$ if $J$ is finitely generated.

If $ I^\ast =J^\ast$ for
 some finitely generated ideal $J$, we say that $I$ is $\ast$-{\it finite} .

A prime ideal which is also a $\ast$-ideal is called a \emph{$\ast$-prime}; a \emph{$\ast$-maximal ideal} is a
$\ast$-ideal maximal in the set of proper integral $\ast$-ideals of $D$. A $\ast$-maximal ideal is prime.
We denote by $\astmax(D)$ the set of $\ast$-maximal ideals of $D$.
If $\ast$ is a star
operation of finite type, by Zorn's lemma each $\ast$-ideal is contained in
a $\ast$-maximal ideal and we have $D=\bigcap\{D_{M}\,;\; M\in \astmax(D)\}$.
We say that $D$
has \emph{$\ast$-finite character} if each nonzero element of $D$ is contained in
at most finitely many $\ast$-maximal ideals.

When $\ast $ is of finite type, a minimal prime of a $\ast$-ideal is
a $\ast$-prime. So, any minimal prime over a nonzero
principal ideal (in particular any height-one prime) is a $\ast$-prime, for any star operation $\ast$ of finite type.
We say that $D$ has \emph{$\ast$-dimension
one} if each $\ast$-prime ideal has height one.

 The \emph{identity} is a star operation denoted by $d$, $I^{d}:=I$
for each $I\in {\mc F} (D)$.

The $v$-\emph{operation}, or \emph{divisorial closure}, of $I\in {\mc F} (D)$ is defined by setting
$$I^{v}:=(D:(D:I)),$$
where for any $I$, $J\in {\mc F} (D)$ we set $(J\colon I):=\{x\in K\,;\; xI
\subseteq J\}$. A $v$-ideal of $D$ is also called a \emph{divisorial ideal}.
It is not difficult to check that
$$I^v=\bigcap \{yD\,;\; y\in K\,,\;I\sub yD\}.$$

The $t$-\emph{operation} is the star operation of finite type associated to $v$ and is therefore defined by setting
$$I^t:= \bigcup \{J^v \,;\; J
\in {\mc F}(D) \mbox{ finitely generated and } J \subseteq I\}.$$

Another star operation of finite type associated to a star operation $\ast$, often denoted by $\widetilde{\ast}$, is
defined by setting
$I^{\widetilde{\ast}}:= \bigcap\{ID_M\,;\;M \in \ast_f\text{-Max}(D)\}$.
It follows easily from the definition that  $\widetilde{\ast}=\widetilde{\ast_f}$ and $\widetilde{\ast}\text{-Max}(D)=\ast _f\text{-Max}(D)$.

 The star operation $\widetilde{v}=\widetilde{t}$ is usually denoted by $w$;
thus the $w$-\emph{operation}  is defined by
setting
$$I^{w} := \bigcap\{ID_{M}\,;\; M\in t\text{-Max}(D)\}$$
for each nonzero ideal $I$.

An equivalent definition of the $w$-operation is obtained by setting
 $$I^{w} := \bigcup \{(I\colon  J) ; \, J \mbox{ is finitely generated and } (D\colon J) = D\}.$$
By using the latter definition, one can see that the notion of $w$-ideal coincides
with the notion of {\it semi-divisorial} ideal introduced by  Glaz
and Vasconcelos in 1977 \cite{GV}.

As a star-operation, the $w$-operation
was first considered by  Hedstrom and  Houston in
$1980$ under the name of {\it $F_{\infty}$-operation} \cite{HH}.

It is well known that $I_w\sub I_t\sub I_v$, for each nonzero ideal $I$.

 For any star operation $\ast$, the set of $\ast$-ideals of $D$, denoted by $\mc F^\ast(D)$, is a
semigroup under {\it $\ast$-multiplication}, defined by $(I, J)\mapsto (IJ)^\ast$, with
unit $D$.
An ideal $I\in \mc F (D)$ is called {\it $\ast$-invertible} if $I^\ast$ is
invertible in $\mc F^\ast(D)$, equivalently $(I(D:I))^\ast=D$.

The quotient semigroup  $\mc S^\ast(D):=\mc F^\ast(D)/\mc P(D)$, where $\mc P(D)$ is the group of nonzero principal ideals of $D$, is called the \emph{$\ast$-class semigroup of $D$}.

\subsection{Polynomial rings and Nagata rings}

If $X$ is an indeterminate over $D$ and $I$ is an ideal of $D$, we set $I[X]:=ID[X]$. We will use repeatedly the following well-known properties.

\begin{lemma} \label{lemmapoli} Let $D$ be an integral domain and $\ast= w, t, v$. Denote by $\ast^\prime$ the respective star operation in $D[X]$. Then:
\begin{itemize}
\item[\rm(1)]  $(J:I)[X]=(J[X]:I[X])$, for any nonzero ideals  $I,J$  of $D$. In particular $I^\ast[X]=(I[X])^{\ast^\prime}$.

\item[\rm(2)] There is an inclusion preserving injective correspondence $I\mapsto I[X]$ between the set of integral $\ast$-ideals of $D$ and the set of integral $\ast^\prime$-ideals of $D[X]$, whose left inverse is the intersection. In addition, $I$ is $\ast$-invertible if and only if $I[X]$ is $\ast^\prime$-invertible.

\item[\rm(3)]  Let $N$ be a $t^\prime$-maximal ideal of $D[X]$.
 Then either $N=M[X]$, with
$M:=N\cap D\in \tmax(D)$ or $N=fK[X]\cap D[X]$ for some irreducible polynomial $f\in K[X]$.
In the second case, $D[X]_N=K[X]_{fK[X]}$ is a $DVR$.
\end{itemize}
\end{lemma}
\begin{proof} (1) The proof of the first statement is exactly as the proof given in \cite[Proposition 4.1]{HH} for $J=D$. The second statement is \cite[Proposition 4.3]{HH}.

(2) The correspondence $I\mapsto I[X]$ is injective because $I=I[X]\cap D$. If $I$ is $\ast$-invertible,  its extension $I[X]$ is $\ast^\prime$-invertible.  Conversely,
if $D[X]=(I[X](D[X]:I[X]))^{\ast^\prime}$, then by item (1)  $D[X]=(I(D:I))^\ast[X]$ and $D=(I(D:I))^\ast$ by the injectivity.

(3) If $N\cap D:= M\neq (0)$, $N=M[X]$ by \cite[Proposition 1.1]{HZ} and
$M$ is $t$-maximal by \cite[Lemma 2.1]{UMT}.
Assume that $N\cap D=(0)$. Since $K[X]=D[X]_S$ is a ring of fractions of $D[X]$, with $S:=D\sm \{0\}$,  then $NK[X]=N_S$ is a prime ideal of $K[X]$. Hence $NK[X]= fK[X]$, with $f\in K[X]$ irreducible and $N=N_S\cap D[X]=fK[X]\cap D[X]$. Finally, $D[X]_N=(D[X]_S)_{N_S}=K[X]_{fK[X]}$ is $DVR$.
\end{proof}

A tool very useful in the study of polynomial ring is given by the so-called Nagata rings of $D$, which are particular rings of fractions of $D[X]$.
As in \cite{K}, for a star
operation $\ast$ on $D$, we set $N(\ast):= \{h(X)\in D[X] \mid h(X) \neq 0 \;\hbox{and}\;
c(h)^{\ast}=D\}$, where $c(f)$ is the \emph{content} of the polynomial $f(X)$, that is the ideal of $D$ generated by the coefficients of $f(X)$. Since $c(h)^\ast=D$ if and only if $c(h)\nsub M$, for each $M\in \ast_f\op{-Max}(D)$, we see that
$N(\ast)=N(\ast_f)=N(\widetilde{\ast})=D[X]\sm\bigcup_{M\in\ast_f\op{-Max}(D)}M$.

The domain $\Na(D, \ast):=D[X]_{N(\ast)}$ is called the \emph{Nagata
ring of $D$ with respect to $\ast$}. For $\ast=d$, $\Na(D,d)=: D(X)$ is the usual
\emph{Nagata ring of $D$} \cite[Section 33]{g1}.

\begin{proposition} \label{Nagata} Let $D$ be an integral domain. Then:
\begin{enumerate}
\item  $\Na(D, v)=\Na(D, t)=\Na(D, w)$.
\item  $\op{Max}(\Na(D, v))=\{M\Na(D, v) \mid M\in \tmax(D)\}$.
\item If $Q= P\Na(D, v)$, with $P\in \spec(D)$, then $\Na(D, v)_{Q}= D[X]_{PD[X]}=
D_{P}(X)$.
\item  $\Na(D,v) = \bigcap\{D[X]_{MD[X]}, M\in \tmax(D)\}=\bigcap\{D_M(X), M\in \tmax(D)\}$.
\item $\Na(D,v)$ is a DW-domain.
\end{enumerate}
\end{proposition}
\begin{proof} (1) is  \cite[Corollary 3.5]{FL}.
(2) is \cite[Proposition 2.1 (2)]{K}.
(3) is clear, since $\Na(D,v)$ is a ring of fraction of $D[X]$ and $Q =P\Na(D, v)=P[X]\Na(D,v)$.
(4) is \cite[Proposition 3.1]{FL}.
(5) Because each maximal ideal of $\Na(D, v)$ is a $t$-ideal \cite[Proposition 2.1 and Corollary 2.3]{K}.
\end{proof}

The Nagata ring $D(X)$ inherits from $D$ more properties than $D[X]$ and has a better behavior, as the next result shows.

\begin{proposition} \label{D(X)} Let $D$ be an integral domain. Then:
\begin{enumerate}
\item $D$ is Noetherian if and only if $D(X)$ is Noetherian;
\item $D$ is Pr\"ufer if and only if $D(X)$ is a Bezout domain;
\item $D$ is Dedekind if and only if $D(X)$ is a PID.
\end{enumerate}\end{proposition}
\begin{proof}
(1) see for example \cite[p.430]{giampaNoeth}.
(2) is \cite[Theorem 33.4]{g1}.
(3) is \cite[Proposition 38.7]{g1}.
\end{proof}

\section{$w$-Stability and  Clifford $w$-regularity of polynomial rings}

Let $S$ be a commutative multiplicative semigroup. An element $x\in S$ is called \emph{von Neuman regular} (for short, \emph{vN-regular})  if there exists an element $a\in S$ such that $x=x^2a$. Idempotent and invertible elements are vN-regular.
By a well-known theorem of Clifford, $S$ is a disjoint union of groups if and only if  all its elements are vN-regular: in this case, $S$ is called a \emph{Clifford semigroup}.

A domain $D$  is called a \emph{Clifford regular domain} if its class semigroup $\mc S(D)$ is Clifford regular.

Dedekind domains are trivial examples of Clifford regular domains.
 Bazzoni and Salce showed that all valuation domains are Clifford regular and gave a complete description of the structure of  $\mc S(D)$ in that case  \cite{BS}.
Zanardo and Zannier investigated the class semigroups of orders in number fields and showed that all orders in quadratic fields are Clifford regular domains \cite{ZZ}. The study of Clifford regularity was then carried on by Bazzoni \cite{B1, B2, B3, B4}.

A particular class of Clifford regular domains is given by stable domains.

We recall that a nonzero ideal
$I$ of $D$ is said to be \emph{stable} if it is invertible in the overring $E(I):=(I:I)$ of $D$, which is the endomorphism ring of $I$. A domain is (\emph{finitely}) \emph{stable} if each (finitely generated) ideal  is invertible in its endomorphism ring. Stable
domains have been thoroughly investigated by  Olberding  \cite{O3, O5, O1, O2}.

When $I$ is stable, we have $I(E(I):I)=E(I)$, so that $I=IE(I)=I^2(E(I):I)$ is vN-regular. It follows that stable domains are Clifford regular.
Conversely, not all Clifford regular domains are stable: in fact,  a valuation domain is always Clifford regular \cite{BS}, but it is stable if and only if  $P\neq P^2$, for each nonzero prime ideal $P$ \cite[Proposition 4.1]{O3}. On the other hand,  Clifford regular domains are finitely stable, so that in the Noetherian case Clifford regularity coincides with stability  \cite[Theorem 3.1]{B3}.

Stability with respect to star operations (and more generally to semistar operations) was introduced and studied by the authors of this paper in \cite{GP}.

The first attempt to extend the notion of Clifford regularity in the setting of star operations is due to Kabbaj and Mimouni, who considered the $t$-operation \cite{KM1, KM2, KM3, KM4}. Then
Halter-Koch, in the language of ideal systems, introduced Clifford $\ast$-regularity for star operations of finite type \cite{HK}.
Finally, we deepened the  study of stability and Clifford regularity with respect to star operations in \cite{GP1, GP2}. 

 We note that Clifford $w$-regularity implies Clifford $t$-regularity, but it is a  stronger property. For example, while any Noetherian Clifford $w$-regular domain has $t$-dimension one, there are Noetherian $t$-regular domains of $t$-dimension greater than two \cite[Section 3]{GP2}.
 
Stability and Clifford regularity with respect to the $w$-operation are defined in the following way. Set as usual $E(J):=(J:J)$, for each  $J\in \mc F(D)$.
It is easy to
see that $E(I^w)^w=E(I^w)$, for each $I\in \mc F(D)$. Thus the
restriction of $w$ to the set of nonzero ideals of
$E(I^w)$ is a star operation on $E(I^w)$,
denoted by $\dot{w}:=w_{\vert_E}$.

 We say that
a nonzero ideal $I$ of $D$ is \emph{$w$-stable} if $I^w$ is
$\dot{w}$-invertible in $E(I^w)$ and that $D$ is
\emph{$w$-stable}  if each ideal of $D$ is $w$-stable.

We also say that $D$ is \emph{Clifford $w$-regular} if the $w$-class semigroup $\mc S^w(D):=\mc F^w(D)/\mc P(D)$  is a Clifford semigroup, i.e., for each nonzero ideal $I$, the class $[I^w]\in \cS^w(D)$ is vN-regular. This is equivalent to saying that $I^w$ is vN-regular in $\F^w(D)$, that is $I^w=(I^2J)^w$, for some nonzero ideal $J$ of $D$; in this case, we also have $I^w=(I^2(E(I^w):I))^w=(I^2(I^w:I^2))^w$ \cite[Lemma 1.2]{GP1}.
 If  $I^w$ is vN-regular in $\F^w(D)$, we say for short  that $I$ is  \emph{$w$-regular}. Clearly, a $w$-stable ideal is $w$-regular and so a $w$-stable domain is Clifford $w$-regular.

 Finally, recall that stable and Clifford regular domains are DW-domains \cite[Corollary 1.11]{GP}, \cite[Corollary 1.7]{GP1}; thus polynomial rings over a domain that is not a field are never stable or Clifford regular (Theorem \ref{th1}).
Here we are interested in the transfer of $w$-stability and Clifford $w$-regularity to polynomial rings and Nagata rings.

We start by observing that a necessary condition for $D[X]$ being Clifford $w$-regular  or $w$-stable is that $D$ has the same property.

 \begin{proposition}\label{stabledown}
    If $D[X]$ is a Clifford $w$-regular (resp., $w$-stable) domain, then $D$ is a Clifford $w$-regular (resp., $w$-stable) domain.
    \end{proposition}
    \begin{proof}
    It is an easy application of Lemma \ref{lemmapoli}, (1) and (2).
\end{proof}

The following proposition shows that Clifford $w$-regularity and $w$-stability of polynomial rings and Nagata rings depend on the non-extended ideals. On the other hand, recall that each integral ideal of $\Na(D,v)$ is extended from $D$ if and only if $D$ is a P$v$MD \cite[Theorem 3.1]{K}.

\begin{lemma} \label{flat} Let $D\sub T$ be a flat extension of domains. If $I$ is a $w$-regular (resp., $w$-stable) ideal of $D$, $IT$ is a  $w$-regular (resp., $w$-stable) ideal of $T$.
\end{lemma}
\begin{proof} This follows from \cite[Lemma 2.4]{GP1}, because a flat extension is $w$-compatible.
\end{proof}

\begin{proposition}\label{extended} If $D$ is a Clifford $w$-regular (resp., $w$-stable) domain, each extended ideal of $D[X]$ is $w$-regular (resp.,
 $w$-stable) and each extended ideal of $\Na(D,v)$ is vN-regular (resp., stable).
\end{proposition}
 \begin{proof}  It follows from Lemma \ref{flat}, because polynomial rings and localizations are flat extensions, and from the fact that $\Na(D,v)$ is a DW-domain (Proposition \ref{Nagata}).  \end{proof}

 The study of ($w$-)stability can be reduced to the local case. In fact a domain $D$ is ($w$-)stable if and only if $D_M$ is stable, for each ($t$-)maximal ideal $M$, and $D$ has ($t$-)finite character \cite[Corollary 1.10]{GP}. If $D$ is Clifford ($w$-)regular, then $D_M$ is Clifford regular for each ($t$-)maximal ideal $M$ \cite[Corollary 2.13]{GP1} and $D$ has ($t$-)finite character \cite[Theorem 5.2]{GP1}, but it is not known if the converse is true in general. However, the converse holds if $D$ is integrally closed or if each nonzero ($t$-)prime ideal of $D$ is contained in a unique ($t$-)maximal ideal (e.g., $D$ has ($t$-)dimension one). This follows from more general results proved in \cite{GP1} in the setting of star operations spectral and of finite type. We give below a direct proof.

Recall that  a Pr\"ufer domain (resp., a P$v$MD) is called \emph{strongly discrete} if $P\neq P^2$ for each prime (resp., $t$-prime) ideal $P$.
A domain is integrally closed and Clifford $w$-regular (resp., $w$-stable) if and only if it is a P$v$MD (resp., a strongly discrete P$v$MD) with $t$-finite character \cite[Corollary 4.5]{GP1}, \cite[Therem 2.9]{GP}. Thus, for $w=d$, an integrally closed Clifford regular (resp., stable) domain is precisely a Pr\"ufer domain (resp., a strongly discrete Pr\"ufer domain) with finite character \cite[Theorem 4.5]{B3}, \cite[Theorem 4.6]{O3}.

If each prime ideal of $D$ is contained in a unique ($t$-)maximal ideal and $D$ has
($t$-)finite character,
$D$ is called \emph{$h$-local} (\emph{weakly Matlis}).

\begin{theorem} \label{FC3}  \cite[Proposition 5.4 and Theorem 5.6]{GP1} Let $D$ be an integral domain. Assume that
 {\rm (a)} $D$ is integrally closed or  {\rm (b)} each nonzero ($t$-)prime ideal of $D$ is contained in a unique ($t$-)maximal ideal (e.g., $D$ has ($t$-)dimension one).  The following conditions are equivalent:
\begin{enumerate}
\item[(i)] $D$ is Clifford ($w$-)regular;
\item[(ii)] $D_M$ is Clifford regular for each $M\in(t\op{-)Max}(D)$, and $D$ has ($t$-)finite character.
\end{enumerate}
\end{theorem}
\begin{proof} It is enough to prove the theorem for Clifford $w$-regularity.

(i) $\ra$ (ii) is always true. In fact, when $D$ is Clifford $w$-regular, $D_M$ is Clifford regular for each $M\in (t\op{-)Max}(D)$ by \cite[Proposition 2.8]{B3}, and $D$ has  $t$-finite character by \cite[Theorem 5.2]{GP1}.

(ii) $\ra$ (i) Assume (a). Since $D_M$ is integrally closed and Clifford regular, it is a valuation domain  \cite[Theorem 3]{BS}. Thus $D$ is Clifford $w$-regular by \cite[Corollary 4.5]{GP1}.

Assume (b). By definition, $D$ is weakly Matlis. Hence $(I:I^2)_M=(I_M:I^2_M)$, for each  $w$-ideal $I$ and $t$-maximal ideal $M$ \cite[Corollary 5.2]{AZ}. Then, since $D$ is $t$-locally Clifford regular,
$$(I^2(I:I^2))^w=\bigcap_{M\in t\op{-Max}(D)}I^2(I:I^2)_M=\bigcap_{M\in t\op{-Max}(D)}I^2_M(I_M:I^2_M)=\bigcap_{M\in t\op{-Max}(D)}I_M=I.$$
We conclude that each nonzero ideal of $D$ is Clifford $w$-regular.
\end{proof}

\begin{corollary} \label{wM} \cite[Corollaries 2.16 and  2.17]{GP1} Assume that  $D$ is  an $h$-local (resp.,  weakly Matlis) domain. Then $D$ is  Clifford ($w$-)regular  if and only if $D_M$ is Clifford regular, for each ($t$-)maximal ideal $M$ of $D$. \end{corollary}

\begin{remark} \rm In the proof of (ii) $\ra$ (i) of Theorem \ref{FC3}, the fact that $D$ is $h$-local (resp., weakly Matlis) is used only because this implies that $(I:I^2)_M=(I_M:I^2_M)$ for every ($w$-)ideal $I$ and ($t$-)maximal ideal $M$. But this is true also in other cases, for example when $I^2$ is finitely generated (resp., $w$-finite).

By a standard argument, the condition that each $w$-ideal is $w$-finite is equivalent to the ascending chain condition on integral $w$-ideals: a domain with this condition is called a \emph{strong Mori domain}. Clearly Noetherian domains are strong Mori. More precisely, a domain is strong Mori if and only if $D_M$ is Noetherian, for each $M\in \tmax(D)$ and $D$ has $t$-finite character \cite[Theorem 1.9]{W}.

It follows that a strong Mori domain $D$ is Clifford $w$-regular if and only if $D_M$ is Clifford regular, for each $M\in \tmax(D)$ \cite[Proposition 2.5(2)]{GP2}.
\end{remark}

Taking in account Theorem \ref{FC3} and Corollary \ref{wM}, to reduce ourselves to consider the local case, we now want to establish when a polynomial ring over a Clifford $w$-regular domain has $t$-finite character or is weakly Matlis.

\begin{lemma}\label{finchar} \cite[Lemma 2.1]{GHP} Let $D$ be an integral domain. The following conditions are equivalent:
	\begin{itemize}
\item[(i)] $D$ has $t$-finite character;
\item[(ii)] $D[X]$ has $t$-finite character;
\item[(iii)] $\Na(D, v)$ has finite character.
		\end{itemize}
\end{lemma}

By Lemmas \ref{lemmapoli}  and \ref{finchar},  when $D$ is weakly Matlis, $D[X]$ is weakly Matlis if and only if each upper to zero is contained in a unique $t$-maximal ideal  \cite[Proposition 2.2]{GHP}. This condition is certainly satisfied if $D$ is a local domain whose
maximal ideal is a $t$-ideal or if each upper to zero
is a $t$-maximal ideal.  A domain such that each upper to zero
is a $t$-maximal ideal is called a \emph{UMT-domain} \cite{HZ}.

	        \begin{proposition} \label{c:umt} \cite[Corollary 2.3]{GHP} Assume that {\rm (a)}  $D$ is a local domain whose
maximal ideal is a $t$-ideal or {\rm (b)}  $D$ is a UMT-domain. The
following conditions are equivalent:
\begin{itemize}
\item[(i)] $D$ is weakly Matlis;
\item[(ii)] $D[X]$ is weakly Matlis;
\item[(iii)] $\Na(D, v)$  is $h$-local.
\end{itemize}
\end{proposition}

It is worth noting that the hypotheses (a) and (b) are not necessary to prove the equivalence of conditions (ii) and (iii) of Proposition \ref{c:umt}.
Examples of weakly Matlis domains $D$ such that $D[X]$ is not weakly Matlis were given in \cite[Examples 2.5 and 2.6]{GHP}.

  \begin{proposition} \label{UMT} A Clifford $w$-regular domain is a UMT-domain.
\end{proposition}
\begin{proof} If $D$ is Clifford $w$-regular, $D_M$ is Clifford regular for each $M\in \tmax(D)$ \cite[Corollary 2.13]{GP1}.  Thus $D_M$ has Pr\"ufer integral closure \cite[Proposition 2.3]{KM1} and hence $D$ is a UMT-domain by \cite[Theorem 1.5]{UMT}. \end{proof}

\begin{corollary} \label{UMT1} Let  $D$ be a Clifford $w$-regular domain. Then:
\begin{itemize}
\item[(1)] If $D$ is weakly Matlis,  the polynomial ring $D[X]$ is weakly Matlis.
\item[(2)] If $D$ has  $t$-dimension one, $D$ is weakly Matlis and the polynomial ring $D[X]$ is weakly Matlis of $t$-dimension one.
\end{itemize}
\end{corollary}
\begin{proof} (1) Since a Clifford $w$-regular domain is a UMT-domain (Proposition \ref{UMT}), $D[X]$ is weakly Matlis  by Proposition \ref{c:umt}.

(2) Let $D$ be of $t$-dimension one. Since a Clifford $w$-regular domain has $t$-finite character \cite[Theorem 5.2]{GP1}, $D$ is weakly Matlis and so
 $D[X]$ is weakly Matlis by item (1). Besides,
each upper to zero of $D[X]$ is $t$-maximal (of height one). Thus each extended prime $t$-ideal has height one and we conclude that $D[X]$ has $t$-dimension one.  \end{proof}

\begin{remark} \rm
A domain is called \emph{quasi-Pr\"ufer} if its integral closure is a Pr\"ufer domain \cite[Corollary 6.5.14]{FHP}. Clearly, the quasi-Pr\"ufer property does not pass to  polynomial rings, since the integral closure of a polynomial ring is a polynomial ring and so it is not Pr\"ufer. However, it is known that a domain is quasi-Pr\"ufer if and only if it is a UMT DW-domain \cite[Theorem 2.4]{DHLRZ} and that the UMT-property transfers to  polynomial rings \cite[Theorem 2.4]{UMT}. This is still another example of a class of domains whose ``non-DW" part passes to polynomials.
\end{remark}

The following result, due to Olberding, is useful to understand what happens when $D$ is local and stable.

\begin{theorem} \label{Olb} \cite[Theorem 2.3]{O1} A domain $D$ is stable if and only if  {\rm (a)} $D$ is finitely stable; {\rm (b)} $QD_Q$ is a stable ideal of $D_Q$, for each nonzero prime ideal $Q$; {\rm (c)} $D_Q$ is a valuation domain, for each nonzero nonmaximal prime ideal $Q$; {\rm (d)} $D$ has finite character.
\end{theorem}

\begin{proposition}\label{local}
	Let $D$ be a  local stable domain. The following conditions are equivalent:
	\begin{itemize}
		\item[(i)]  $D(X)$ is stable;
		 \item[(ii)]  $D(X)$ is Clifford regular;
	\item[(iii)]  $D(X)$ is finitely stable;
	\item[(iv)] Each non-extended finitely generated ideal of $D(X)$ is stable.	
	\end{itemize}
	\end{proposition}
	\begin{proof} (i) $\ra$ (ii) and (iii) $\ra$ (iv) are clear. (ii) $\ra$ (iii) by \cite[Proposition 2.3]{B2}.
	
	(iv) $\ra$ (i) We apply Theorem \ref{Olb}.
	Let $M$ be the maximal ideal of $D$.
The domain $R:=D(X):=\Na(D, v)$ is  local, with maximal ideal $M(X)$.
Since $D$ is a UMT-domain (Lemma \ref{UMT}), each prime ideal of $R$, is extended from $D$ \cite[Theorem  3.1]{HZ}.
Also, since $D$ is stable, by Proposition \ref{extended}, each extended ideal of $R$ is stable. Hence, each nonzero prime ideal $Q$ of $R$ is stable and so $QR_Q$ is stable.  In addition, if the prime ideal $Q:=P(X)$ of $R$ is not maximal,  $P$ is not maximal. Hence $D_P$ is a valuation domain and $R_Q=D_P(X)$ is also a valuation domain.
In conclusion, if each non-extended finitely generated ideal of $R$ is stable, $R$ is finitely stable and $R$ satisfies all the conditions of  Theorem \ref{Olb}. Thus $R$ is stable.
	\end{proof}

By applying Proposition \ref{local}, we now show that the $w$-stability of polynomial rings depends on the stability of certain local Nagata rings.

\begin{proposition} \label{nagatastable}
Let $D$ be an integral domain. The following conditions are equivalent:
\begin{enumerate}
\item[(i)] $D[X]$ is $w$-stable;
\item[(ii)] $\Na(D,v)$ is stable;
\item[(iii)] $D[X]$ has  $t$-finite character and $D_{M}(X)$ is stable for each $M\in \tmax(D)$.
\end{enumerate}
Under (any one of) these conditions, $D$ is $w$-stable.
\end{proposition}
\begin{proof}
(i) $\ra$ (ii) Since $\Na(D,v)$ is a ring of fractions of $D[X]$, every ideal of $\Na(D,v)$ is extended from $D[X]$. Hence $\Na(D,v)$ is $w$-stable by Proposition \ref{extended}. Moreover, $\Na(D,v)$ is a DW-domain by 
Proposition \ref{Nagata}, so it is stable.

(ii) $\ra$ (iii) The $t$-maximal ideals of $D[X]$ are either uppers to zero or extended from $D$ (Lemma \ref{lemmapoli}). If $N$ is an upper to zero, $D[X]_N$ is a DVR. Otherwise, $N=MD[X]$ with $M\in \tmax(D)$. By Proposition \ref{Nagata}, $M\Na(D,v)$ is a maximal ideal of $\Na(D,v)$ and $\Na(D,v)_{M\Na(D,v)}=D[X]_{MD[X]}=D_M(X)$ is  stable. Thus $D[X]$ is  $t$-locally stable.  In addition, since $\Na(D,v)$ has finite character, $D[X]$ has $t$-finite character (Lemma \ref{finchar}).

(iii) $\ra$ (i) We apply the $t$-local characterization \cite[Corollary 1.10]{GP}.

Under condition (i), $D$ is $w$-stable by Proposition \ref{stabledown}.
\end{proof}

	\begin{theorem}\label{stable} Assume that $D$ is a $w$-stable domain. The following
conditions are equivalent:
	\begin{itemize}
		\item[(i)]  $D[X]$ is $w$-stable;
		\item[(ii)]  $\Na(D, v)$ is stable;
		 \item[(iii)] $D_{M}(X)$ is stable for each $M\in \tmax(D)$;
\item[(iv)] $D_{M}(X)$ is Clifford regular for each $M\in \tmax(D)$.
 \end{itemize}
\end{theorem}
\begin{proof}
(i) $\lra$ (ii) $\lra$ (iii) Since $w$-stable domains have $t$-finite character \cite[Corollary 1.10]{GP}, $D[X]$ has $t$-finite character (Lemma \ref{finchar}). Hence we can apply Proposition \ref{nagatastable}.

(iii) $\lra$ (iv) by Proposition \ref{local}, because $D_M$ is stable \cite[Corollary 1.10]{GP}.
\end{proof}

In order to get a similar result for Clifford $w$-regularity, we need some additional hypotheses.

\begin{proposition}\label{nagataclifford}
Let $D$ be an integral domain and consider the following conditions:
\begin{enumerate}
\item[(1)] $D[X]$ is Clifford $w$-regular;
\item[(2)] $\Na(D,v)$ is Clifford regular;
\item[(3)] $D[X]$ has $t$-finite character and $D_{M}(X)$ is Clifford regular for each $M\in \tmax(D)$.
\end{enumerate}
Then {\rm (1)} $\ra$ {\rm (2)} $\ra$ {\rm (3)}. Moreover, if $D$ is integrally closed or $D[X]$ is weakly Matlis, then {\rm (3)} $\ra$ {\rm (1)} and $D$ is Clifford $w$-regular.
\end{proposition}
\begin{proof}
The proofs of (1) $\ra$ (2) and of (2) $\ra$ (3) are exactly the same as the proofs of Proposition \ref{nagatastable} ((i)$\ra$ (ii) and (ii) $\ra$ (iii) respectively) with Clifford ($w$-)regular instead of $w$-stable.

If (3) holds, we cannot conclude in general that $D[X]$ is Clifford $w$-regular. However, under the additional hypothesis that $D$ (and so $D[X]$) is integrally closed or $D[X]$ is weakly Matlis, we can apply  Theorem \ref{FC3} and  obtain that $D[X]$ is Clifford $w$-regular.

Finally, under condition (1), $D$ is Clifford $w$-regular by Proposition \ref{stabledown}.
\end{proof}

 \begin{theorem} \label{tone} Assume that the domain $D$ is Clifford $w$-regular and weakly Matlis.
 The following
conditions are equivalent:
	\begin{itemize}
		\item[(i)]  $D[X]$ is  Clifford $w$-regular;
		\item[(ii)]  $\Na(D, v)$ is  Clifford regular;
 \item[(iii)] $D_{M}(X)$ is Clifford regular for each $M\in \tmax(D)$.
 \end{itemize}
	\end{theorem}
	\begin{proof} $D[X]$ is weakly Matlis by Corollary \ref{UMT1}(1), in particular it has $t$-finite character. Hence we can apply Proposition \ref{nagataclifford}.
	\end{proof}
	
	Putting together Theorems \ref{stable} and \ref{tone}, we immediately get that if $D$ is $w$-stable and weakly Matlis, $w$-stability and Clifford $w$-regularity of polynomial rings are equivalent.
	
	\begin{theorem} \label{Mori2} Let $D$ be a $w$-stable weakly Matlis domain. The following conditions are equivalent:
	\begin{itemize}
		\item[(i)]  $D[X]$ is $w$-stable;
			\item[(ii)]  $D[X]$ is Clifford $w$-regular;
		\item[(iii)]Ê$\Na(D, v)$ is stable;
			\item[(iv)]Ê$\Na(D, v)$ is Clifford regular;
			\item[(v)] $D_{M}(X)$ is stable for each $M\in \tmax(D)$;
			\item[(vi)] $D_{M}(X)$ is Clifford regular for each $M\in \tmax(D)$.
	\end{itemize}
		\end{theorem}
		 \begin{proof}  		
		(i) $\lra$ (iii) $\lra$ (v) $\lra$ (vi) follow from Theorem \ref{stable}.
		
		(ii) $\lra$ (iv) $\lra$ (vi) by Theorem \ref{tone}.		
				\end{proof}
	
	A class of weakly Matlis domains is given by \emph{$w$-divisorial domains}, that is, domains in which each $w$-ideal is divisorial \cite{GE}. In fact, $D$ is $w$-divisorial if and only if $D$ is weakly Matlis and $D_M$ is divisorial, for each $t$-maximal ideal $M$ \cite[Theorem 1.5]{GE}.  (A \emph{divisorial domain} is a domain whose ideals are all divisorial.) The transfer of $w$-divisoriality to polynomial rings was studied in \cite{GHP}.
	
A domain 	$D$ that is at the same time $w$-stable and $w$-divisorial is precisely a weakly Matlis domain such that $D_M$ is totally divisorial, for each $t$-maximal ideal $M$ \cite[Corollary 3.2]{GP}. (A \emph{totally divisorial domain} is a domain whose overrings are all divisorial.)
				
	\section{Two positive results}
	
	We are now able to show that Clifford $w$-regularity and $w$-stability pass to polynomial rings in two cases. The first one is when $D$ is integrally closed, the second case is when $D$ is a  $w$-stable $w$-divisorial Mori domain.

 \begin{theorem} \label{icstable} Let $D$ be an integrally closed domain. The following conditions are equivalent:
 \begin{itemize}
 \item[(i)] $D$ is Clifford $w$-regular (resp., $w$-stable);
  \item[(ii)] $D[X]$ is Clifford $w$-regular (resp., $w$-stable);
   \item[(iii)] $\Na(D,v)$ is Clifford regular (resp., stable).
   \end{itemize}
 \end{theorem}
 \begin{proof} (i) $\ra$ (iii) Since an integrally closed Clifford $w$-regular domain is a P$v$MD \cite[Corollary 4.5]{GP1}, each ideal of $\Na(D,v)$ is extended from $D$ \cite[Theorem 3.1]{K}. So, by Proposition \ref{extended}, if $D$ is Clifford  $w$-regular (resp.,  $w$-stable), $\Na(D,v)$ is Clifford  regular (resp.,  stable).

 (iii) $\ra$ (ii) It is always true for stability (Proposition \ref{nagatastable}) and it is true under the ``integrally closed'' hypothesis for Clifford regularity (Proposition \ref{nagataclifford}).

 (ii) $\ra$ (i)  is Proposition \ref{stabledown}.
  \end{proof}

	Recall that a \emph{Mori domain} is a domain with the ascending chain condition on integral divisorial ideals.
	For the main properties of Mori domains, one can see the survey  \cite{Bar} and the references there.

Since divisorial ideals are $w$-ideals,  strong Mori domains (i.e., domains satisfying the ascending chain condition on integral $w$-ideals) are Mori. More precisely, it follows from \cite[Theorem 1.9]{W} that
$D$ is strong Mori if and only if $D$ is Mori and $D_M$ is Noetherian for each $M\in \tmax(D)$.

Stability and Clifford regularity of Mori domains were studied in \cite{GP2}, in the more general setting of star operations; in particular, it was proved there  that $w$-stability and Clifford $w$-regularity coincide for strong Mori domains \cite[Corollary 3.11]{GP2}. But these two notions  are indeed equivalent for all Mori domains; this follows from a recent result  in \cite{GR}.

\begin{theorem} \label{morionedim}  \cite{GR}
Let $D$ be an integral domain. The following conditions are equivalent:
\begin{enumerate}
\item[(i)] $D$ is stable and one-dimensional;
\item[(ii)] $D$ is Mori and stable;
\item[(iii)] $D$ is Mori and finitely stable.
\end{enumerate}
\end{theorem}

A class of examples of  local one-dimensional domains which are stable and not Noetherian  has been constructed  by Olberding   \cite[Theorems 4.1 and 4.4]{O6} (see also \cite[Theorem 3.10]{O7}); by Theorem \ref{morionedim}, all these domains must be Mori.

	 \begin{theorem} \label{Mori1} Let $D$ be an integral domain.
 The following
conditions are equivalent:
	\begin{itemize}
		\item[(i)]  $D$ is  $w$-stable of $t$-dimension one;
		\item[(ii)]  $D$ is  Mori and $w$-stable;
 \item[(iii)] $D$ is Mori and Clifford $w$-regular.
  \end{itemize}
	\end{theorem}
\begin{proof} (i) $\ra$ (ii) For each $M\in \tmax(D)$, $D_M$ is stable and one-dimensional; hence it is  Mori by Theorem \ref{morionedim}. By $w$-stability, $D$ has $t$-finite character \cite[Corollary 1.10]{GP} and this implies that $D$ is Mori \cite[Theorem 2.4]{Bar}.

(ii) $\ra$ (iii)
is clear.

(iii) $\ra$ (i) For each $M\in \tmax(D)$, $D_M$ is Mori and Clifford regular; hence finitely stable \cite[Proposition 2.3]{B3}. By Theorem \ref{morionedim}, $D_M$ is stable. Since a Mori domain has $t$-finite character \cite[Theorem 3.3]{Bar},  $D$ is $w$-stable  \cite[Corollary 1.10]{GP}.
\end{proof}

Since $w$-stable domains of $t$-dimension one are weakly Matlis (Corollary \ref{UMT1}(2)),  Theorem \ref{Mori2} can be restated for Mori domains.

\begin{theorem} \label{cormori} Let $D$ be a $w$-stable Mori domain. The following conditions are equivalent:
	\begin{itemize}
		\item[(i)]  $D[X]$ is  $w$-stable;
		\item[(ii)] $D[X]$ is  Clifford $w$-regular;			
		\item[(iii)] $D_{M}(X)$ is stable for each $M\in \tmax(D)$.
	\end{itemize}
	In addition, under (any one of) these conditions, $D[X]$ is a Mori domain.
	\end{theorem}
	\begin{proof} A $w$-stable Mori domain has $t$-dimension one (Theorem \ref{Mori1}) and so it is  weakly Matlis by Corollary \ref{UMT1}(2). Hence Theorem \ref{Mori2} holds for $w$-stable Mori domains.
	
In addition, if $D$ is $w$-stable of $t$-dimension one, also $D[X]$ has $t$-dimension one (Corollary \ref{UMT1}(2)). Thus,  if $D[X]$ is $w$-stable, $D[X]$ is Mori  by Theorem \ref{Mori1}.
	\end{proof}
	
Hence, when $D$ is Mori,  a necessary condition for $D[X]$ being $w$-stable is that  $D[X]$ is Mori.
	Even though a polynomial ring over a Mori domain need not be Mori \cite[Section 6]{Bar}, we do not know what happens in $t$-dimension one. However, if $D$ is strong Mori (i.e., Mori and $t$-locally Noetherian), also $D[X]$ is strong Mori  \cite[Theorem 1.13]{W}. Thus, in view of Theorem \ref{cormori},  we ask:
	
	\begin{Qu} Let $D$ be a stable local Noetherian domain. Is $D(X)$ stable? \end{Qu}
	
	The answer to this question is positive under the additional hypothesis that $D$ is  divisorial (i.e., each nonzero ideal is divisorial). It follows from the definitions that a  Mori divisorial domain is Noetherian.
	
		We recall that a  domain $D$ is stable and divisorial if and only if it is totally divisorial (i.e., each overring of $D$ is divisorial) \cite[Theorem 3.12]{O5}. When $D$ is  Mori,  $D$ is  stable and divisorial (equivalently, totally divisorial) if and only it is Noetherian \emph{2-generated}, that is, each ideal can be generated by two elements \cite[Theorem 3.1]{O5}. There are several examples of stable Noetherian domain that are not 2-generated  \cite[Section 3]{O7}; a first example was given in \cite[Example 5.4]{SV74}.
			
	 \begin{proposition} \label{2gen}
 Let $D$ be a local Noetherian domain. The following
conditions are equivalent:
	\begin{itemize}
		\item[(i)]  $D$ is  totally divisorial;
		\item[(ii)]  $D$ is  $2$-generated;
		\item[(iii)]  $D(X)$ is $2$-generated;
 \item[(iv)] $D(X)$ is  totally divisorial.
  \end{itemize} \end{proposition}
 \begin{proof}
 Note that $D(X)$ is Noetherian by Proposition \ref{D(X)}.

 (i) $\lra$ (ii) and (iii) $\lra$ (iv) by \cite[Theorem 3.1]{O5}.

 (ii) $\lra$ (iii)
$D$ and $D(X)$ have the same multiplicity \cite[p.214 and Lemma 8.4.2(6)]{swanson} and  a local one-dimensional Noetherian domain has multiplicity $2$ if and only if it is $2$-generated \cite[Lemma 3.1]{SV74}.
 \end{proof}

 When $D$ is Mori,  $D$ is $w$-stable and  $w$-divisorial if and only $D_M$ is Noetherian and totally divisorial, for each $M\in \tmax(D)$  \cite[Corollary 3.6]{GP1}.
 In particular, a $w$-divisorial $w$-stable Mori domain is strong Mori.

By  \cite[Proposition 3.6]{GHP}, if $D$ is Mori and $w$-divisorial, $D[X]$ is $w$-divisorial.
 We  now show that, if in addition $D$ is $w$-stable,  $D[X]$ is also  $w$-stable.

 \begin{theorem} \label{Mori3} Let $D$ be a Mori domain. The following conditions are equivalent:
	\begin{itemize}
        \item[(i)] $D$ is $w$-stable and $w$-divisorial;
		\item[(ii)]  $D[X]$ is $w$-stable and $w$-divisorial;
        \item[(iii)] $Na(D,v)$ is totally divisorial;
		\item[(iv)]  $D_M(X)$ is 2-generated (equivalently, Noetherian totally divisorial), for each $M\in \tmax(D)$.
\end{itemize}
\end{theorem}
\begin{proof}
(i) $\Leftrightarrow$ (iv) $D$ is $w$-stable and $w$-divisorial if and only if, for every  $t$-maximal ideal $M$ of $D$, $D_M$ is Noetherian and totally divisorial \cite[Corollary 3.6]{GP1},  if and only if $D_M(X)$ is $2$-generated for every $t$-maximal ideal $M$ of $D$ (Proposition \ref{2gen}).

(iv)+(i) $\ra$ (ii) $D[X]$ is $w$-divisorial by \cite[Proposition 3.6]{GHP} and it is $w$-stable by Theorem \ref{cormori}.

(ii) $\ra$ (i) $D$ is $w$-stable by Proposition \ref{stabledown} and $w$-divisorial by \cite[Proposition 3.6]{GHP}.

(ii) $\Leftrightarrow$ (iii) It follows from Proposition \ref{nagatastable}, \cite[Proposition 3.2]{FHP} and \cite[Theorem 3.12]{O5}.
\end{proof}

The problem of establishing whether more generally $w$-stability of Mori domains transfers to polynomial rings can be similarly reduced  to the local case (Theorem \ref{cormori}); that is, it can be reduced to investigate  the stability of $D(X)$ when $D$ is a local stable Mori domain, equivalently, a local stable one-dimensional domain (Theorem \ref{morionedim}). Our previous results and a theorem of Olberding show that one has only to consider the case when the conductor of the integral closure $D^\prime$ is zero.

\begin{theorem} \label{SVOlb} \cite[Proposition 4.5]{O1}  Let $D$ be a one-dimensional stable domain.  If $(D:D^\prime)\neq (0)$, $D$ is 2-generated and $D^\prime$ is a finitely generated $R$-module.
\end{theorem}

\begin{corollary}  Let $D$ be a Mori $w$-stable domain.  If $(D:D^\prime)\neq (0)$,  $D$ is $w$-divisorial and $D[X]$ is $w$-stable and $w$-divisorial.
\end{corollary}
\begin{proof} For each $M\in \tmax(D)$, $D_M$ is one-dimensional stable (Theorem \ref{morionedim}) and its integral closure has nonzero conductor. Hence $D_M$ is 2-generated (Theorem \ref{SVOlb}), that is $D_M$ is totally divisorial \cite[Theorem 3.1]{O5}.
By \cite[Theorem 4.5]{GE}, $D$ is $w$-divisorial and so $D[X]$ is $w$-stable and $w$-divisorial (Theorem \ref{Mori3}).
\end{proof}

Explicit examples of local one-dimensional stable or  2-generated domains such that $(D:D^\prime) = (0)$ can be found in \cite[Section 3]{O7}.


\end{document}